\documentclass[12pt]{amsart}
\usepackage{amsmath,amsthm,amsfonts,amssymb,mathrsfs}
\date{\today}

\usepackage{color}
\input xymatrix
\xyoption{all}

\usepackage[colorlinks,plainpages,citecolor=magenta, linkcolor=blue, backref]{hyperref}%,bookmarksnumbered,

\usepackage{hyperref}

\newtheorem{theorem}{Theorem}[section]

\newtheorem{proposition}[theorem]{Proposition}
\newtheorem{corollary}[theorem]{Corollary}
\newtheorem{lemma}[theorem]{Lemma}
\theoremstyle{definition}
\newtheorem{example}[theorem]{Example}%[section]
\newtheorem{remark}[theorem]{Remark}%[section]

\newtheorem{definition}[theorem]{Definition}%[section]

\begin{document}

\title[On semigroups which admit only discrete left-continuous Hausdorff topology]{On semigroups which admit only discrete left-continuous Hausdorff topology}

\author{Oleg~Gutik}
\address{Ivan Franko National University of Lviv,
Universytetska 1, Lviv, 79000, Ukraine}
\email{oleg.gutik@lnu.edu.ua}



\begin{abstract}
We give the sufficient condition when every left-continuous (right-continuous) Hausdorff topology on a semigroup $S$ is discrete. We construct a submonoid $\mathscr{C}_{+}(a,b)$ (resp., $\mathscr{C}_{-}(a,b)$) of the bicyclic monoid which contains a family $\{S_\alpha\colon \alpha\in\mathfrak{c}\}$ of continuum many subsemigroups with the following properties: $(i)$ every left-continuous (resp., right-continuous) Hausdorff topology on $S_\alpha$ is discrete; $(ii)$ every semigroup $S_\alpha$ admits a non-discrete right-continuous (resp., left-continuous) Hausdorff topology which is not left-continuous (resp., right-continuous); $(iii)$ every semigroup $S_\alpha$ isomorphically embeds into a Hausdorff compact topological semigroup. Also we construct a submonoid $\mathscr{C}_{\mathbb{Z}}^+$ (resp., $\mathscr{C}_{\mathbb{Z}}^-$) of the extended bicyclic semigroup which contains a family $\{S_\alpha\colon \alpha\in\mathfrak{c}\}$ of continuum many subsemigroups with the above described properties.
\end{abstract}

\maketitle

\section{\textbf{Introduction, motivation and main definitions}}\label{section-1}

In this paper we shall follow the terminology of \cite{Carruth-Hildebrant-Koch=1983, Clifford-Preston=1961,  Engelking=1989, Ruppert=1984}.

%\smallskip

By $\omega$ we denote the set of all non-negative integers and by $\mathbb{Z}$ the set of all integers. Throughout these notes we always assume that all topological spaces involved are Hausdorff -- unless explicitly stated otherwise.

\begin{definition}
Let $X$, $Y$ and $Z$ be topological spaces. A map $f\colon X\times Y\to Z$, $(x,y)\mapsto f(x,y)$, is called
\begin{itemize}
  \item[$(i)$] \emph{right} [\emph{left}] \emph{continuous} if it is continuous in the right [left] variable; i.e., for every fixed $x_0\in X$ [$y_0\in Y$] the map $Y\to Z$, $y\mapsto f(x_0,y)$ [$X\to Z$, $x\mapsto f(x,y_0)$] is continuous;
  \item[$(ii)$] \emph{separately continuous} if it is both left and right continuous;
  \item[$(iii)$] \emph{jointly continuous} if it is continuous as a map between the product space $X\times Y$ and the space $Z$.
\end{itemize}
\end{definition}

\begin{definition}[\cite{Carruth-Hildebrant-Koch=1983, Ruppert=1984}]
Let $S$ be a non-void topological space which is provided with an associative multiplication (a semigroup operation) $\mu\colon S\times S\to S$, $(x,y)\mapsto \mu(x,y)=xy$. Then the pair $(S,\mu)$ is called
\begin{itemize}
  \item[$(i)$] a \emph{right topological semigroup} if the map $\mu$ is right continuous, i.e., all interior left shifts $\lambda_s\colon S\to S$, $x\mapsto sx$, are continuous maps, $s\in S$;
  \item[$(ii)$] a \emph{left topological semigroup} if the map $\mu$ is left continuous, i.e., all interior right shifts $\rho_s\colon S\to S$, $x\mapsto xs$, are continuous maps, $s\in S$;
  \item[$(iii)$] a \emph{semitopological semigroup} if the map $\mu$ is separately continuous;
  \item[$(iv)$] a \emph{topological semigroup} if the map $\mu$ is jointly continuous.
\end{itemize}

We usually omit the reference to $\mu$ and write simply $S$ instead of $(S,\mu)$. It goes without saying that every topological semigroup is
also semitopological and every semitopological semigroup is both a right and left topological semigroup.
\end{definition}

A topology $\tau$ on a semigroup $S$ is called:
\begin{itemize}
  \item a \emph{semigroup} topology if $(S,\tau)$ is a topological semigroup;
  %\item an \emph{inverse semigroup} topology if $(S,\tau)$ is a topological inverse semigroup;
  \item a \emph{shift-continuous} topology if $(S,\tau)$ is a semitopological semigroup;
  \item an \emph{left-continuous} topology if $(S,\tau)$ is a left topological semigroup;
  \item an \emph{right-continuous} topology if $(S,\tau)$ is a right topological semigroup.
\end{itemize}

If $S$ is a semigroup then by $E(S)$ we denote the set of all idempotents of $S$.

The \emph{bicyclic monoid} ${\mathscr{C}}(a,b)$ is the semigroup with the identity $1$ generated by two elements $a$ and $b$ subjected only to the condition $ab=1$. The semigroup operation on ${\mathscr{C}}(a,b)$ is determined as
follows:
\begin{equation*}\label{eq-1}
    b^ka^l\cdot b^ma^n=
    \left\{
      \begin{array}{ll}
        b^{k-l+m}a^n, & \hbox{if~} l<m;\\
        b^ka^n,       & \hbox{if~} l=m;\\
        b^ka^{l-m+n}, & \hbox{if~} l>m.
      \end{array}
    \right.
\end{equation*}
It is well known that the bicyclic monoid ${\mathscr{C}}(a,b)$ is a bisimple (and hence simple) combinatorial $E$-unitary inverse semigroup and every non-trivial congruence on ${\mathscr{C}}(a,b)$ is a group congruence \cite{Clifford-Preston=1961}.

On $\mathbf{B}_\omega=\omega\times\omega$ we define the semigroup operation ``$\cdot$'' in the following way
\begin{equation}\label{eq-1.1}
  (k,l)\cdot(m,n)=
  \left\{
      \begin{array}{ll}
        (k-l+m,n), & \hbox{if~} l<m;\\
        (k,n),     & \hbox{if~} l=m;\\
        (k,l-m+n), & \hbox{if~} l>m,
      \end{array}
    \right.
\end{equation}
$k,l,m,n\in\omega$. It is obvious that the bicyclic monoid ${\mathscr{C}}(a,b)$ is isomorphic to the semigroup $\mathbf{B}_\omega$ by the mapping  $b^ia^j\mapsto (i,j)$. The set $\mathscr{C}_{\mathbb{Z}}=\mathbb{Z}\times\mathbb{Z}$ with the semigroup operation \eqref{eq-1.1} is called the \emph{extended bicyclic semigroup} \cite{Warne=1968}.

A semigroup $S$ is called \emph{inverse} if for any $s\in S$ there exists the unique $t\in S$ such that $sts=s$ and $tst=t$. In this case the element $t$ is called \emph{inverse of} $s$ and it is denoted by $s^{-1}$. Every inverse semigroup $S$ admits the \emph{natural partial order}:
\begin{equation*}
  s\preccurlyeq t \qquad \hbox{if and only if there exists an idempotent } \quad e\in S \quad \hbox{such that} \quad s=et.
\end{equation*}
It is well known that the bicyclic monoid and the extended bicyclic semigroup are inverse semigroup. Also the natural partial order $\preccurlyeq$ on $\mathscr{C}_{\mathbb{Z}}$ and $\mathbf{B}_\omega$ is determined in the following way
\begin{equation*}
  (k,l)\preccurlyeq(m,n) \Longleftrightarrow (k-l=m-n \;\hbox{and} \; m\leqslant k) \Longleftrightarrow (k-l=m-n \;\hbox{and} \; n\leqslant l).
\end{equation*}
%\smallskip

For a semigroups $S$ and $T$ a map $\alpha\colon S\to S$ is said to be an \emph{anti-homomorphism} if $\alpha(s\cdot t)=\alpha(t)\cdot \alpha(s)$. A bijective anti-homomorphism of semigroups is called an \emph{anti-isomorphism}.

%\smallskip

It is well known that topological algebra studies the influence of topological properties of its objects on their algebraic properties and the influence of algebraic properties of its objects on their topological properties. There are two main problems in topological algebra: the problem of non-discrete topologization and the problem of embedding into objects with some topological-algebraic properties.

%\smallskip

In mathematical literature the question about non-discrete (Hausdorff) topologization was posed by Markov in \cite{Markov=1945}. Pontryagin gave well known conditions on a base at the unity of a group for its non-discrete topologization (see Theorem~4.5 of \cite{Hewitt-Roos=1963}). Various authors have refined Markov's question: \emph{can a given infinite group $G$ endowed with a non-discrete group topology be embedded into a compact topological group?} Again, for an arbitrary Abelian group $G$ the answer is affirmative, but there is a non-Abelian topological group that cannot be embedded into any compact
topological group ({see Section~9 of \cite{HBSTT-1984}}).

%\smallskip

Also, Ol'shanskiy \cite{Olshansky=1980} constructed an infinite countable group $G$ such that every Hausdorff group topology on $G$ is discrete. In \cite{Klyachko-Olshanskii-Osin=2013} the authors constructed first examples of infinite hereditarily non-topologizable groups.
In \cite{Zelenyuk=2000} Ye.~Zelenyuk proved that on every countable group there exists a nondiscrete regular topology with continuous shifts and inversion. Also, every countable locally finite group admits a non-discrete Hausdorff topology \cite{Belyaev=1995}.
Taimanov presented in \cite{Taimanov=1973} a commutative semigroup $\mathfrak{T}$ which admits only discrete Hausdorff semigroup topology. Also in \cite{Taimanov=1975} he gave sufficient conditions on a commutative semigroup to have a non-discrete semigroup topology. In \cite{Gutik=2016} it is proved that each shift-continuous  $T_1$-topology on the Taimanov semigroup $\mathfrak{T}$ is discrete.
The bicyclic monoid admits only the discrete shift-continuous Hausdorff topology \cite{Eberhart-Selden=1969, Bertman-West=1976}. In \cite{Chornenka-Gutik=2023} non-discrete shift-continuous (semigroup, semigroup inverse) $T_1$-topologies on the bicyclic monoid ${\mathscr{C}}(a,b)$ are constructed. In \cite{Chornenka-Gutik=2023} are presented the sufficient conditions under which a shift-continuous $T_1$-topology on ${\mathscr{C}}(a,b)$ is discrete. Also every shift-continuous Hausdorff topology on the extended bicyclic semigroup and any interassociate of the bicyclic monoid is discrete \cite{Fihel-Gutik=2011, Gutik-Maksymyk=2016}.

%\smallskip

Stable and $\Gamma$-compact topological semigroups do not contain the bicyclic monoid~\cite{Anderson-Hunter-Koch=1965, Hildebrant-Koch=1986, Koch-Wallace=1957}. The problem of embedding the bicyclic monoid into compact-like topological semigroups was studied in \cite{Banakh-Dimitrova-Gutik=2009, Banakh-Dimitrova-Gutik=2010, Bardyla-Ravsky=2020, Gutik-Repovs=2007}.

%\smallskip

%Subsemigroups of the bicyclic monoid are studied in \cite{Descalco-Ruskuc=2005, Descalco-Ruskuc=2008, Hovsepyan=2020}.

%\smallskip

In \cite{Makanjuola-Umar=1997}  the following anti-isomorphic subsemigroups of the bicyclic monoid
\begin{equation*}
  \mathscr{C}_{+}(a,b)=\left\{b^ia^j\in\mathscr{C}(a,b)\colon i\leqslant j\right\}
  \quad \hbox{and} \quad
  \mathscr{C}_{-}(a,b)=\left\{b^ia^j\in\mathscr{C}(a,b)\colon i\geqslant j\right\}
\end{equation*}
are studied. In the paper \cite{Gutik=2023} we prove that every Hausdorff left-continuous (right-continuous) topology on the monoid $\mathscr{C}_{+}(a,b)$  ($\mathscr{C}_{-}(a,b)$) is discrete and show that there exists a compact Hausdorff topological monoid $S$ which contains $\mathscr{C}_{+}(a,b)$  ($\mathscr{C}_{-}(a,b)$) as a submonoid. Also, in \cite{Gutik=2023} we constructed a non-discrete right-continuous (left-continuous) topology $\tau_p^+$ ($\tau_p^-$) on the semigroup $\mathscr{C}_{+}(a,b)$ ($\mathscr{C}_{-}(a,b)$) which is not left-continuous (right-continuous).

Similarly as in \cite{Makanjuola-Umar=1997} the following two
\begin{equation*}
  \mathscr{C}^{+}_\mathbb{Z}=\left\{(i,j)\in\mathscr{C}_\mathbb{Z}\colon i\leqslant j\right\}
  \quad \hbox{and} \quad
  \mathscr{C}^{-}_\mathbb{Z}=\left\{(i,j)\in\mathscr{C}_\mathbb{Z}\colon i\geqslant j\right\}
\end{equation*}
of $\mathscr{C}_\mathbb{Z}$. The semigroup operation of $\mathscr{C}_\mathbb{Z}$ implies that $\mathscr{C}^{+}_\mathbb{Z}$ and $\mathscr{C}^{-}_\mathbb{Z}$ are anti-isomorphic subsemigroups of $\mathscr{C}_\mathbb{Z}$.

%\smallskip

In the paper we give the sufficient condition when every left-continuous (right-continuous) Hausdorff topology on a semigroup $S$ is discrete. We construct a submonoid $\mathscr{C}_{+}(a,b)$ (resp., $\mathscr{C}_{-}(a,b)$) of the bicyclic monoid which contains a family $\{S_\alpha\colon \alpha\in\mathfrak{c}\}$ of continuum many subsemigroups with the following properties: $(i)$ every left-continuous (resp., right-continuous) Hausdorff topology on $S_\alpha$ is discrete; $(ii)$ every semigroup $S_\alpha$ admits a non-discrete right-continuous (resp., left-continuous) Hausdorff topology which is not left-continuous (resp., right-continuous); $(iii)$ every semigroup $S_\alpha$ isomorphically embeds into a Hausdorff compact topological semigroup. Also we construct a submonoid $\mathscr{C}_{\mathbb{Z}}^+$ (resp., $\mathscr{C}_{\mathbb{Z}}^-$) of the extended bicyclic semigroup which contains a family $\{S_\alpha\colon \alpha\in\mathfrak{c}\}$ of continuum many subsemigroups with the above described properties.

%\smallskip

\section{On topologizations of subsubsemigroups of  the semigroups $\mathscr{C}_{+}(a,b)$ and $\mathscr{C}_{-}(a,b)$}

The following example shows that the bicyclic monoid ${\mathscr{C}}(a,b)$ admits a Hausdorff right-continuous (left-continuous) topology $\tau_p^r$ ($\tau_p^l$) such that left (right) translations are not continuous in $(\mathscr{C}(a,b),\tau_p^r)$ ($(\mathscr{C}(a,b),\tau_p^l)$).

\begin{example}\label{example-2.1}
Fix an arbitrary prime positive integer $p$.  For any  $b^ia^j\in{\mathscr{C}}(a,b)$ and any $n\in\omega$ we denote
\begin{equation*}
U_n^r(b^ia^j)=\left\{b^ia^{j+p^n\cdot k}\colon k\in \omega\right\}.
\end{equation*}
Then the family $\mathscr{B}_p^r=\left\{\mathscr{B}_p^r(b^ia^j)\colon  i,j\in\omega\right\}$, where $\mathscr{B}_p^r(b^ia^j)=\left\{U_n^r(b^ia^j)\colon n\in\omega\right\}$ for $b^ia^j\in{\mathscr{C}}(a,b)$, satisfies the properties (BP1)--(BP4) of \cite{Engelking=1989}, and hence it generates a Hausdorff topology $\tau_p^r$ on the bicyclic monoid $\mathscr{C}(a,b)$. Obviously that $\tau_p^r$ is a non-discrete topology $\mathscr{C}(a,b)$, and moreover all points of $(\mathscr{C}(a,b),\tau_p^r)$ are non-isolated.

We claim that $\tau_p^r$ is a right-continuous topology on $\mathscr{C}(a,b)$. Indeed, for any $b^{i_1}a^{j_1},b^{i_2}a^{j_2}\in\mathscr{C}(a,b)$ and any $n\in\omega$ we have that
\begin{align*}
  b^{i_1}a^{j_1}\cdot U_n^r(b^{i_2}a^{j_2})&=\left\{b^{i_1}a^{j_1}\cdot b^{i_2}a^{j_2+p^n\cdot k}\colon k\in\omega\right\}= \\
   &=
   \left\{
     \begin{array}{ll}
       b^{i_1-j_1+i_2}a^{j_2+p^n\cdot k}, & \hbox{if~} j_1<i_2;\\
       b^{i_1}a^{j_2+p^n\cdot k},         & \hbox{if~} j_1=i_2;\\
       b^{i_1}a^{j_1-i_2+j_2+p^n\cdot k}, & \hbox{if~} j_1>i_2
     \end{array}
   \right.=\\
    &=
   \left\{
     \begin{array}{ll}
       U_n^r(b^{i_1-j_1+i_2}a^{j_2}), & \hbox{if~} j_1<i_2;\\
       U_n^r(b^{i_1}a^{j_2}),         & \hbox{if~} j_1=i_2;\\
       U_n^r(b^{i_1}a^{j_1-i_2+j_2}), & \hbox{if~} j_1>i_2,
     \end{array}
   \right.
\end{align*}
and hence the semigroup operation is right-continuous in $(\mathscr{C}(a,b),\tau_p^r)$.

Fix arbitrary  $b^{i_1}a^{j_1},b^{i_2}a^{j_2}\in\mathscr{C}(a,b)$ such that ${j_1}<{i_2}$. Then we have that $b^{i_1}a^{j_1}\cdot b^{i_2}a^{j_2}=b^{i_1-j_1+i_2}a^{j_2}$. For any $n\in\omega$ there exists $k_0$ such that the neighbourhood $U_n^r(b^{i_1}a^{j_1})$ contains the elements $b^{i_1}a^{j_1+p^n\cdot k}$ with the property that $j_1+p^n\cdot k>i_2$ for all $k>k_0$. This implies that $b^{i_1}a^{j_1+p^n\cdot k}\cdot b^{i_2}a^{j_2}=b^{i_1}a^{j_1+p^n\cdot k-i_2+j_2}$ and hence we get that
\begin{equation*}
  U_n^r(b^{i_1}a^{j_1}) \cdot b^{i_2}a^{j_2}\nsubseteq U_m^r(b^{i_1-j_1+i_2}a^{j_2})
\end{equation*}
for any $n,m\in\omega$. Thus the semigroup operation is not left-continuous in $(\mathscr{C}(a,b),\tau_p^r)$.

The anti-isomorphism of the bicyclic monoid $\mathfrak{AI}\colon \mathscr{C}(a,b)\to\mathscr{C}(a,b)$, $b^ia^j\mapsto b^ja^i$, $i,j\in\omega$, generated a left-continuous topology on $\mathscr{C}(a,b)$ in the following way. For any  $b^ia^j\in{\mathscr{C}}(a,b)$ and any $n\in\omega$ we denote
\begin{equation*}
U_n^l(b^ia^j)=(U_n^r(b^ja^i))\mathfrak{AI}
\end{equation*}
Then the family $\mathscr{B}_p^l=\left\{\mathscr{B}_p^r(b^ia^j)\colon b^ia^j\in{\mathscr{C}}(a,b), n\in\omega\right\}$, where $\mathscr{B}_p^l(b^ia^j)=\left\{U_n^r(b^ia^j)\colon n\in\omega\right\}$, satisfies the properties (BP1)--(BP4) of \cite{Engelking=1989}, and hence it generates a non-discrete Hausdorff topology $\tau_p^l$ on the bicyclic monoid $\mathscr{C}(a,b)$. Since $\mathfrak{AI}\colon \mathscr{C}(a,b)\to\mathscr{C}(a,b)$ is the anti-isomorphism, the semigroup operation is left-continuous in $(\mathscr{C}(a,b),\tau_p^r)$ but it is not right-continuous.
\end{example}

The following theorem gives the sufficient conditions on a semigroup $S$ under which every Hausdorff left-contionuous (right-continuous) topology on $S$ is discrete.

\begin{theorem}\label{theorem-2.1}
Let $S$ be an infinite semigroup. If for any $s\in S$ there exists an idempotent $e_s\in S$ such that $s\in S\setminus Se_s$ $(s\in S\setminus e_sS)$ and the set $S\setminus Se_s$ $(S\setminus e_sS)$ is finite, then every Hausdorff left-contionuous (right-continuous) topology on $S$ is discrete.
\end{theorem}

\begin{proof}
Fix an arbitrary $s\in S$.
Obviously that for the idempotent $e_s$ of the semigroup $S$ the right translation $\rho_{e_s}\colon S\to S$, $s\mapsto se_s$ is a continuous retraction of the topological space $S$. Hausdorffness of $S$ and Exercise 1.5.C of \cite{Engelking=1989} imply that $Se_s$ is a closed subspace of $S$. Since the set $S\setminus Se_s$ is finite and open, the point $s$ is isolated in $S$.

The proof of the dual statement is similar.
\end{proof}

Theorem~\ref{theorem-2.1} implies the following corollary.

\begin{corollary}\label{corollary-2.2}
Let  $S$ be a subsemigroup of $\mathscr{C}_{+}(a,b)$ $(\mathscr{C}_{-}(a,b))$ such that the set $E(S)$ is infinite. then every Hausdorff left-contionuous (right-continuous) topology on $S$ is discrete.
\end{corollary}

\begin{proof}
Fix an arbitrary $b^{i}a^{i+k}\in S$, $i,k\in\omega$. Since the set $E(S)$ is infinite, there exists a positive integer $p$ such that $p>k$ and $b^{i+p}a^{i+p}\in E(S)$. The semigroup operation of $\mathscr{C}_{+}(a,b)$ (and hence of $S$) implies that for any $b^sa^t\in S$ we have that
\begin{equation*}
  b^sa^t\cdot b^{i+p}a^{i+p}=
  \left\{
    \begin{array}{ll}
      b^{s-t+i+p}a^{i+p}, & \hbox{if~} t<i+p;\\
      b^{s}a^{i+p},       & \hbox{if~} t=i+p; \\
      b^sa^t,             & \hbox{if~} t>i+p.
    \end{array}
  \right.
\end{equation*}
By the above equalities we get that $b^{i}a^{i+k}\notin S\cdot b^{i+p}a^{i+p}$, and moreover the set $\mathscr{C}_{+}(a,b)\setminus (\mathscr{C}_{+}(a,b)\cdot b^{i+p}a^{i+p})$ is finite. This implies that $S\setminus S\cdot b^{i+p}a^{i+p}$ is finite, as well. Next we apply Theorem~\ref{theorem-2.1}.

In the case of semigroup $\mathscr{C}_{-}(a,b)$ the proof is similar.
\end{proof}

The following example and later statements show that the semigroups $\mathscr{C}_{+}(a,b)$ and $\mathscr{C}_{-}(a,b)$ contains continuum many subsemigroups which satisfy the assumptions of Theorem~\ref{theorem-2.1}.

\begin{example}\label{example-2.3}
For an arbitrary non-negative integer $k$ we define
\begin{equation*}
  \mathscr{C}_{+k}(a,b)=\left\{b^{i}a^{i+s}\in\mathscr{C}_{+}(a,b)\colon s\geqslant k, s\in \omega\right\}.
\end{equation*}
The semigroup operation of $\mathscr{C}_{+}(a,b)$ implies that $\mathscr{C}_{+k}(a,b)$ is a subsemigroup of $\mathscr{C}_{+}(a,b)$. Fix an arbitrary
infinite subset $X$ of $\omega$. Latter we shall assume that $X=\left\{x_i\colon i\in\omega\right\}$ where $\left\{x_i\right\}_{i\in\omega}$ is a
steadily increasing sequence in $\omega$. Put
\begin{equation*}
  \mathscr{C}_{+k}^{X}(a,b)=\mathscr{C}_{+k}(a,b)\cup\left\{b^{x_i}a^{x_i}\in\mathscr{C}_{+}(a,b)\colon i\in\omega \right\}.
\end{equation*}
Since
\begin{equation*}
  b^ja^{j+s}\cdot b^{x_i}a^{x_i}=
\left\{
  \begin{array}{ll}
    b^{x_i-s}a^{x_i}, & \hbox{if~} j+s<x_i;\\
    b^ja^{j+s},       & \hbox{if~} j+s\geqslant x_i
  \end{array}
\right.
\end{equation*}
and
\begin{equation*}
  b^{x_i}a^{x_i}\cdot  b^ja^{j+s}=
\left\{
  \begin{array}{ll}
    b^ja^{j+s},       & \hbox{if~} x_i<j;\\
    b^{x_i}a^{x_i+s}, & \hbox{if~} x_i\geqslant j
  \end{array}
\right.
\end{equation*}
for all $i,j,s\in \omega$, the inequality $s\geqslant k$ implies that $\mathscr{C}_{+k}^{X}(a,b)$ is a subsemigroup of $\mathscr{C}_{+}(a,b)$. Obviously that the semigroup $\mathscr{C}_{+k}^{X}(a,b)$ contains infinitely many idempotents. Hence the conditions of Corollary~\ref{corollary-2.2} hold for the semigroup $\mathscr{C}_{+k}^{X}(a,b)$.
\end{example}

\begin{lemma}\label{lemma-2.4}
For an arbitrary infinite subset $X$  of $\omega$ the semilattice $(X,\max)$ is isomorphic to the semilattice $(\omega,\max)$.
\end{lemma}

\begin{proof}
Since $(\omega,\max)$ is a linearly ordered semilattice, $(X,\max)$ is a semilattice, too.

Without loss of generality we may assume that $X=\left\{x_i\colon i\in\omega\right\}$ where $\left\{x_i\right\}_{i\in\omega}$ is a steadily increasing sequence in $\omega$. We define the map $\iota_X\colon \omega\to X$ by the formula $(i)\iota_X=x_i$, $i\in\omega$. Since the sets $\omega$ and $X$ are well ordered by the usual linear order $\leqslant$, the map $\iota_X\colon \omega\to X$ is well defined. Obviously that $\iota_X\colon (\omega,\max)\to (X,\max)$ is a semilattice isomorphism. Indeed, since $\left\{x_i\right\}_{i\in\omega}$ is a steadily increasing sequence in $\omega$, the inequality $i\leqslant j$ implies $x_i=(i)\iota_X\leqslant (j)\iota_X=x_j$. It is easy that the map $\iota_X\colon \omega\to X$ is bijective and its inverse is steadily increasing.
\end{proof}

\begin{lemma}\label{lemma-2.5}
Let $\left\{x_i\right\}_{i\in\omega}$ and $\left\{y_i\right\}_{i\in\omega}$ be steadily increasing sequences in $\omega$. Then the sets $X=\left\{x_i\colon i\in\omega\right\}$ and $Y=\left\{y_i\colon i\in\omega\right\}$ with the induced semilattice operation from $(\omega,\max)$ are isomorphic semilattices, and moreover the isomorphism $\iota\colon (X,\max)\to(Y\max)$ is defined by the formula $(x_i)\iota=y_i$, $i\in\omega$.
\end{lemma}

\begin{proof}
By the proof of Lemma~\ref{lemma-2.4} the semilattice isomorphism $\iota_X\colon (\omega,\max)\to (X,\max)$ is defined by the formula $(i)\iota_X=x_i$, $i\in\omega$. These arguments imply that the isomorphism $\iota\colon (X,\max)\to(Y,\max)$ is determined in the following way $(x_i)\iota=((x_i)\iota_X^{-1})\iota_Y=y_i$, $i\in\omega$. Since the isomorphisms $\iota_X$ and $\iota_Y$ are well defined, so is $\iota$.
\end{proof}

\begin{lemma}\label{lemma-2.6}
Let $\left\{x_i\right\}_{i\in\omega}$ and $\left\{y_i\right\}_{i\in\omega}$ be steadily increasing sequences in $\omega$ and let $X=\left\{x_i\colon i\in\omega\right\}$ and $Y=\left\{y_i\colon i\in\omega\right\}$. Then the subsemigroups $\mathscr{C}_{+1}^{X}(a,b)$ and $\mathscr{C}_{+1}^{Y}(a,b)$ of the monoid $\mathscr{C}_{+}(a,b)$ are isomorphic if and only if $x_i=y_i$ for all $i\in\omega$.
\end{lemma}

\begin{proof}
The implication $(\Leftarrow)$ is trivial.

$(\Rightarrow)$ Suppose that the semigroups $\mathscr{C}_{+1}^{X}(a,b)$ and $\mathscr{C}_{+1}^{Y}(a,b)$ are isomorphic. First we observe that the semigroup operation  of $\mathscr{C}_{+}(a,b)$ implies that for any $i\in\omega$ the set of solution of the equality $b^ia^i\cdot z=b^ia^{i+1}$ in $\mathscr{C}_{+}(a,b)$ is the following
\begin{equation*}
  \mathscr{S}_{b^ia^{i+1}}^{b^ia^{i}}=\left\{a,\ldots, b^ia^{i+1}\right\},
\end{equation*}
i.e., the cardinality of the set $\mathscr{S}_{b^ia^{i+1}}^{b^ia^{i}}$ is equal to $i+1$. This implies that the same holds for solutions of the equality $b^{x_i}a^{x_i}\cdot z=b^{x_i}a^{x_i+1}$ in the semigroup $\mathscr{C}_{+1}^{X}(a,b)$ and for solutions of the equality $b^{y_i}a^{y_i}\cdot z=b^{y_i}a^{y_i+1}$ in the semigroup $\mathscr{C}_{+1}^{X}(a,b)$.

If $\mathfrak{I}\colon \mathscr{C}_{+1}^{X}(a,b)\to \mathscr{C}_{+1}^{X}(a,b)$ is an isomorphism, then by Lemma~\ref{lemma-2.5} we get that $(b^{x_i}a^{x_i})\mathfrak{I}=b^{y_i}a^{y_i}$ for any $i\in\omega$. This and above arguments imply that the set of solutions the equality $b^{x_i}a^{x_i}\cdot z=b^{x_i}a^{x_i+1}$ in  $\mathscr{C}_{+1}^{X}(a,b)$ and the set of solutions of the equality $b^{y_i}a^{y_i}\cdot z=b^{y_i}a^{y_i+1}$ in  $\mathscr{C}_{+1}^{X}(a,b)$ have same cardinality $x_i+1=y_i+1$, and hence they coincide. This implies that $x_i=y_i$ for all $i\in\omega$.
\end{proof}

It is well known that the set $\omega$ contains continuum many distinct infinite subsets. This and lemma~\ref{lemma-2.6} imply that there exist continuum many non-isomotrphic subsemigroups in the monoid $\mathscr{C}_{+}(a,b)$ of the forms $\mathscr{C}_{+k}^{X}(a,b)$, where $k\in\omega\setminus\{0\}$ and $X$ is an infinite subset of $\omega$. Then using Theorem~\ref{theorem-2.1} we obtain the following theorem.

\begin{theorem}\label{theorem-2.7}
The monoid $\mathscr{C}_{+}(a,b)$ contains continuum many non-isomorphic subsemigroups of the forms $\mathscr{C}_{+k}^{X}(a,b)$, where $k$ is a positive integer and $X$ is an infinite subset of $\omega$, such that every left-continuous Hausdorff topology on $\mathscr{C}_{+k}^{X}(a,b)$ is discrete.
\end{theorem}

\begin{proof}
Let $X=\left\{x_i\colon i\in\omega\right\}$ be any infinite subset of $\omega$ such that $\left\{x_i\right\}_{i\in\omega}$ is a steadily increasing sequence in $\omega$. For any element $b^ja^{j+s}\in \mathscr{C}_{+k}^{X}(a,b)$ there exists $x_i\in X$ such that $x_i\geqslant j+s+1$. The we have that
\begin{equation*}
  b^ja^{j+s}\notin \mathscr{C}_{+k}^{X}(a,b)\cdot b^{x_i}a^{x_i}\subset \mathscr{C}_{+}(a,b)\cdot b^{x_i}a^{x_i}
\end{equation*}
and the set
\begin{equation*}
  \mathscr{C}_{+k}^{X}(a,b)\setminus\left(\mathscr{C}_{+k}^{X}(a,b)\cdot b^{x_i}a^{x_i}\right)\subset \mathscr{C}_{+}(a,b)\setminus\left(\mathscr{C}_{+}(a,b)\cdot b^{x_i}a^{x_i}\right)
\end{equation*}
is infinite. Next we apply Theorem~\ref{theorem-2.1}.
\end{proof}

By the dual way as in Example~\ref{example-2.3} for an arbitrary infinite subset $X$ of $\omega$ and any positive integer $k$ we construct the subsemigroup $\mathscr{C}_{-k}^{X}(a,b)$ of the monoid $\mathscr{C}_{-}^{X}(a,b)$. We observe that the monoids $\mathscr{C}_{+}(a,b)$ and $\mathscr{C}_{-}(a,b)$ are anti-isomorphic by the mapping $\mathfrak{AI}\colon b^ia^j\mapsto b^ja^i$, $i,j\in\omega$, $i\leqslant j$ (see \cite{Gutik=2023, Makanjuola-Umar=1997}). Simple verifications show that the restriction of $\mathfrak{AI}$ onto $\mathscr{C}_{+k}^{X}(a,b)$ is an anti-isomorphism from $\mathscr{C}_{+k}^{X}(a,b)$ to $\mathscr{C}_{-k}^{X}(a,b)$ for an arbitrary infinite subset $X$ of $\omega$ and any positive integer $k$. The above arguments and Theorem~\ref{theorem-2.7} imply

\begin{theorem}\label{theorem-2.8}
The monoid $\mathscr{C}_{-}(a,b)$ contains continuum many non-isomorphic subsemigroups of the forms $\mathscr{C}_{-k}^{X}(a,b)$, where $k$ is a positive integer and $X$ is an infinite subset of $\omega$, such that every right-continuous Hausdorff topology on $\mathscr{C}_{+k}^{X}(a,b)$ is discrete.
\end{theorem}

We observe that for any positive integer $k$  and any infinite subset $X$ of $\omega$ we have that for every $b^ja^{j}\in \mathscr{C}_{+k}^{X}(a,b)$ there exists its basic neighbourhood $U_n^r(b^ia^j)$ in $(\mathscr{C}(a,b),\tau_p^l)$ such that $U_n^r(b^ia^j)\subset \mathscr{C}_{+k}^{X}(a,b)$, and moreover $U_m^r(b^ia^j)\subset \mathscr{C}_{+k}^{X}(a,b)$ for any positive integer $m>n$. This implies that the topology of the space $(\mathscr{C}(a,b),\tau_p^l)$ induces on $\mathscr{C}_{+k}^{X}(a,b)$ a non-discrete topology. Later we denote this topology on $\mathscr{C}_{+k}^{X}(a,b)$ by $\tau_p^l$. Then the arguments presented in Example~\ref{example-2.1} imply the following proposition.

\begin{proposition}\label{proposition-2.9}
The monoid $\mathscr{C}_{+}(a,b)$ contains continuum many non-isomorphic subsemigroups of the forms $\mathscr{C}_{+k}^{X}(a,b)$, where $k$ is a positive integer and $X$ is an infinite subset of $\omega$, and the semigroup operation is left-continuous and is not right-continuous in $(\mathscr{C}_{+k}^{X}(a,b),\tau_p^l)$.
\end{proposition}

Also, by the dual way we get the following

\begin{proposition}\label{proposition-2.10}
The monoid $\mathscr{C}_{-}(a,b)$ contains continuum many non-isomorphic subsemigroups of the forms $\mathscr{C}_{-k}^{X}(a,b)$, where $k$ is a positive integer and $X$ is an infinite subset of $\omega$, and the semigroup operation is right-continuous and is not left-continuous in $(\mathscr{C}_{-k}^{X}(a,b),\tau_p^r)$.
\end{proposition}

\section{On topologizations of subsubsemigroups of  the semigroups $\mathscr{C}^{+}_{\mathbb{Z}}$ and $\mathscr{C}^{-}_{\mathbb{Z}}$}

For arbitrary $(k,l)\in\mathscr{C}_{\mathbb{Z}}$ we denote
\begin{equation*}
  {\uparrow}_{\preccurlyeq}(k,l)=\left\{(i,j)\in\mathscr{C}_{\mathbb{Z}}\colon (k,l)\preccurlyeq(i,j)\right\} \qquad \mbox{and} \qquad
  {\downarrow}_{\preccurlyeq}(k,l)=\left\{(i,j)\in\mathscr{C}_{\mathbb{Z}}\colon (i,j)\preccurlyeq(k,l)\right\},
\end{equation*}
where $\preccurlyeq$ is the natural partial order on $\mathscr{C}_{\mathbb{Z}}$.

We observe that the semilattice  $E(\mathscr{C}_{\mathbb{Z}})$ of idempotents of $\mathscr{C}_{\mathbb{Z}}$ is isomorphic to the semilattice $(\mathbb{Z},\max)$ by the mappings $(i,i)\mapsto i$, $i\in\mathbb{Z}$.

\begin{lemma}\label{lemma-3.1}
Every Hausdorff shift-continuous topology $\tau$ on $(\mathbb{Z},\max)$ is discrete.
\end{lemma}

\begin{proof}
Fix an arbitrary integer $n$ and suppose that $n$ is a non-isolated point of $(\mathbb{Z},\max,\tau)$. Since the semilattice operation $\max$ on $(\mathbb{Z},\tau)$ is separate continuous, the set ${\uparrow}_{\geqslant}(n-1)=\left\{i\in\mathbb{Z}\colon n-1\geqslant i\right\}$ is closed in the space $(\mathbb{Z},\tau)$ as the full preimage of the point $(n-1)$ under the continuous shift $k\mapsto \max\{k,n-1\}$. This implies that the point $n$ has an open neighbourhood $U(n)$ in $(\mathbb{Z},\tau)$ such that $U(n)\subseteq {\downarrow}_{\geqslant}n=\left\{i\in\mathbb{Z}\colon i\geqslant n\right\}$.

Again, since the shift $k\mapsto \max\{k,n+1\}$ is a continuous retraction of the Hausdorff space $(\mathbb{Z},\tau)$, by Exercise 1.5.C of \cite{Engelking=1989} we get that the set ${\downarrow}_{\geqslant}(n+1)=\left\{i\in\mathbb{Z}\colon i\geqslant n+1\right\}$ is a closed subset of $(\mathbb{Z},\tau)$. Hence $n$ is an isolated point of $(\mathbb{Z},\tau)$, which implies the statement of the lemma.
\end{proof}

\begin{remark}\label{remark-2.2}
\begin{enumerate}
  \item[$(1)$] Since the semilattices $(\mathbb{Z},\max)$ and $(\mathbb{Z},\min)$ are isomorphic by the map $n\mapsto-n$, the statement Lemma~\ref{lemma-3.1} holds for the semilattice $(\mathbb{Z},\min)$.
  \item[$(2)$] The first part of the proof of Lemma~\ref{lemma-3.1} implies that every shift-continuous $T_1$-topology $\tau$ on $(\mathbb{N},\min)$ is discrete.
  \item[$(3)$] The second part of the proof of Lemma~\ref{lemma-3.1} implies that every Hausdorff shift-continuous topology $\tau$ on $(\mathbb{N},\max)$ is discrete. It is obvious that the semilattice $(\mathbb{N},\max)$ admits a non-discrete semigroup $T_1$-topology.
\end{enumerate}
\end{remark}

\begin{lemma}\label{lemma-3.3}
For any integer $k$ the subsemigroup $\mathscr{C}^{+}_{\mathbb{Z}}[k]=\left\{(s,t)\in\mathscr{C}^{+}_{\mathbb{Z}}\colon s\geqslant k\right\}$ of $\mathscr{C}^{+}_{\mathbb{Z}}$ is isomorphic to the semigroup $\mathscr{C}_{+}(a,b)$ by the map $\mathfrak{h}_k\colon \mathscr{C}_{+}(a,b)\to \mathscr{C}^{+}_{\mathbb{Z}}[k]$, $b^ia^j\mapsto (i+k,j+k)$.
\end{lemma}

\begin{proof}
The semigroup operation of $\mathscr{C}_{\mathbb{Z}}$ implies that $\mathscr{C}^{+}_{\mathbb{Z}}[k]$ is a subsemigroup of $\mathscr{C}^{+}_{\mathbb{Z}}$. Since the map $\mathfrak{h}_k$ the equalities
\begin{align*}
  (b^{i_1}a^{j_1}\cdot b^{i_2}a^{j_2})\mathfrak{h}_k &=
  \left\{
    \begin{array}{ll}
      (b^{i_1-j_1+i_2}a^{j_2})\mathfrak{h}_k, & \hbox{if~} j_1<i_2;\\
      (b^{i_1}a^{j_2})\mathfrak{h}_k,         & \hbox{if~} j_1=i_2;\\
      (b^{i_1}a^{j_1-i_2+j_2})\mathfrak{h}_k, & \hbox{if~} j_1>i_2
    \end{array}
  \right.=
  \\
   &=
   \left\{
    \begin{array}{ll}
      (i_1-j_1+i_2+k,j_2+k), & \hbox{if~} j_1<i_2;\\
      (i_1+k,j_2+k),         & \hbox{if~} j_1=i_2;\\
      (i_1+k,j_1-i_2+j_2+k), & \hbox{if~} j_1>i_2
    \end{array}
  \right.
\end{align*}
and
\begin{align*}
  (b^{i_1}a^{j_1})\mathfrak{h}_k\cdot (b^{i_2}a^{j_2})\mathfrak{h}_k &=(i_1+k,j_1+k)\cdot(i_2+k,j_2+k)= \\
   &=
   \left\{
     \begin{array}{ll}
       (i_1+k-(j_1+k)+i_2+k,j_2+k), & \hbox{if~} j_1+k<i_2+k;\\
       (i_1+k,j_2+k),               & \hbox{if~} j_1+k=i_2+k; \\
       (i_1+k,j_1+k-(i_2+k)+j_2+k), & \hbox{if~} j_1+k>i_2+k
     \end{array}
   \right.=
   \\
   &=
   \left\{
    \begin{array}{ll}
      (i_1-j_1+i_2+k,j_2+k), & \hbox{if~} j_1<i_2;\\
      (i_1+k,j_2+k),         & \hbox{if~} j_1=i_2;\\
      (i_1+k,j_1-i_2+j_2+k), & \hbox{if~} j_1>i_2
    \end{array}
  \right.
\end{align*}
imply the statement of the lemma.
\end{proof}

If $\preccurlyeq$ is the natural partial order on $\mathscr{C}_{\mathbb{Z}}$ and $S$ is a subsemigroup of $\mathscr{C}^{+}_{\mathbb{Z}}$ then in this section by $\preccurlyeq^S$ we denote the restriction of the binary relation $\preccurlyeq$ onto $S$, i.e., $(i,j)\preccurlyeq^S(k,l)$ if and only if $(i,j)\preccurlyeq(k,l)$ in $\mathscr{C}_{\mathbb{Z}}$, and put
\begin{equation*}
  {\uparrow}_{\preccurlyeq^S}(m,n)=\left\{(p,q)\in S\colon (m,n)\preccurlyeq^S(p,q)\right\} \quad \mbox{and} \quad {\downarrow}_{\preccurlyeq^S}(m,n)=\left\{(p,q)\in S\colon (p,q)\preccurlyeq^S(m,n)\right\},
\end{equation*}
for any $(m,n)\in S$.

\begin{theorem}\label{theorem-3.4}
Let $S$ be a subsemigroup of $\mathscr{C}^{+}_{\mathbb{Z}}$ $(\mathscr{C}^{-}_{\mathbb{Z}})$ such that the band $E(S)$ of $S$ is isomorphic to the semilattice $(\mathbb{Z},\max)$. Then every Hausdorff left-continuous \emph{(}right-continuous\emph{)} topology $\tau$ on $S$  is discrete.
\end{theorem}

\begin{proof}
If $S=E(S)$ then the statement of the theorem follows from Lemma~\ref{lemma-3.1}. Hence later we assume that  $S\neq E(S)$.

Fix an arbitrary $(i,j)\in S$. Since the band $E(S)$ is isomorphic to the semilattice $(\mathbb{Z},\max)$, the semigroup operation of $\mathscr{C}^{+}_{\mathbb{Z}}$ implies that there exists the maximum integer $k$ such that $(k,k)\in E(S)$ and $k\leqslant i$. By Lemma~\ref{lemma-3.3} the subsemigroup $\mathscr{C}^{+}_{\mathbb{Z}}[k]$ of $\mathscr{C}^{+}_{\mathbb{Z}}$ is isomorphic to the semigroup $\mathscr{C}_{+}(a,b)$, and hence by Corollary~\ref{corollary-2.2}, $S_k=S\cap\mathscr{C}^{+}_{\mathbb{Z}}[k]$ is a discrete subsemigroup of $S$.

Again, since $E(S)$ is isomorphic to $(\mathbb{Z},\max)$, the semigroup operation of $\mathscr{C}^{+}_{\mathbb{Z}}$ implies that there exists the minimum integer $n$ such that $(n,n)\in E(S)$ and $j<n$. Then we have that $(i,j)\cdot (n,n)=(n-j+i,n)$. We observe that
\begin{equation*}
  {\uparrow}_{\preccurlyeq^S}(n-j+i,n)=\left\{(s,t)\in S\colon (s,t)\cdot (n,n)=(n-j+i,n) \right\}=\left\{(n-j+i-p,n-p)\colon p\in\omega\right\}.
\end{equation*}
By Lemma~\ref{lemma-3.3} and Corollary~\ref{corollary-2.2}, the point $(n-j+i,n)$ is isolated in the space $S_n=S\cap\mathscr{C}^{+}_{\mathbb{Z}}[n]$. Then the equality
\begin{equation*}
  (s,t)\cdot(n,n)=
  \left\{
    \begin{array}{ll}
      (n-t+s,n), & \hbox{if~} t\leqslant n;\\
      (s,t),     & \hbox{if~} t>n
    \end{array}
  \right.
\end{equation*}
and the left-continuity of the semigroup operation of $(S,\tau)$ imply that the set ${\uparrow}_{\preccurlyeq^S}(n-j+i,n)$ is open in $(S,\tau)$. It is obvious that $(i,j)\in {\uparrow}_{\preccurlyeq^S}(n-j+i,n)$. If the set ${\uparrow}_{\preccurlyeq^S}(n-j+i,n)$ is finite, then the proof is complete.

Suppose the set ${\uparrow}_{\preccurlyeq^S}(n-j+i,n)$ is infinite. Since $E(S)$ is isomorphic to $(\mathbb{Z},\max)$, the semigroup operation of $\mathscr{C}^{+}_{\mathbb{Z}}$ implies that there exists the maximum integer $s<j$ such that $(s,s)\in E(S)$ and $(i,j)\cdot(s,s)=(i,j)$. We observe that the semigroup operation of $\mathscr{C}^{+}_{\mathbb{Z}}$ implies that the set ${\uparrow}_{\preccurlyeq^S}(i,j)$ is infinite, and hence there exists $(x,y)\in {\uparrow}_{\preccurlyeq^S}(i,j)$ such that $(i,j)\neq (x,y)\cdot(s,s)$. Without loss of generality we may assume that $y<s$ and hence we have that $(x,y)\cdot(s,s)=(s-y+x,s)$. Then Hausdorffness and left-continuity of the semigroup operation of $(S,\tau)$ imply that the set ${\uparrow}_{\preccurlyeq^S}(s-y+x,s)$ is closed in $(S,\tau)$ as the full preimage of the point $(s-y+x,s)$ under the right shift $\rho_{(s,s)}\colon S\to S$, $(a,b)\mapsto (a,b)\cdot (s,s)$. This implies the point $(i,j)$ has a finite open neighbourhood in $(S,\tau)$, and hence is isolated.

The dual statement of the theorem follows from the fact that the semigroups $\mathscr{C}^{+}_{\mathbb{Z}}$ and $\mathscr{C}^{-}_{\mathbb{Z}}$ are antiisomorphic by the mapping $\mathfrak{AI}\colon {\mathscr{C}}^+_\mathbb{Z}\to {\mathscr{C}}^-_\mathbb{Z}$, $(i,j)\mapsto(j,i)$.
\end{proof}

The following example shows that the extended bicyclic semigroup ${\mathscr{C}}_\mathbb{Z}$ admits a Hausdorff right-continuous (left-continuous) topology $\tau_p^r$ ($\tau_p^l$) such that left (right) translations are not continuous in $(\mathscr{C}_\mathbb{Z},\tau_p^r)$ ($(\mathscr{C}_\mathbb{Z},\tau_p^l)$).

\begin{example}\label{example-3.5}
Fix an arbitrary prime positive integer $p$.  For any  $(i,j)\in{\mathscr{C}}_\mathbb{Z}$ and any $n\in\omega$ we denote
\begin{equation*}
U_n^r(i,j)=\left\{(i,j+p^n\cdot k)\colon k\in \omega\right\}.
\end{equation*}
Then the family $\mathscr{B}_p^r=\left\{\mathscr{B}_p^r(i,j)\colon  i,j\in\mathbb{Z}\right\}$, where $\mathscr{B}_p^r(i,j)=\left\{U_n^r(i,j)\colon n\in\omega\right\}$ for $(i,j)\in{\mathscr{C}}_\mathbb{Z}$, satisfies the properties (BP1)--(BP4) of \cite{Engelking=1989}, and hence it generates a Hausdorff topology $\tau_p^r$ on the extended bicyclic semigroup ${\mathscr{C}}_\mathbb{Z}$. Obviously that $\tau_p^r$ is a non-discrete topology $\mathscr{C}_\mathbb{Z}$, and moreover all points of $(\mathscr{C}_\mathbb{Z},\tau_p^r)$ are non-isolated.

The proof of the statements that the semigroup operation on $({\mathscr{C}}_\mathbb{Z},\tau_p^r)$ is right-continuous and it is not left-continuous are similar to the corresponding statements in Example~\ref{example-2.1}. Also, by the anti-isomorphism $\mathfrak{AI}\colon {\mathscr{C}}_\mathbb{Z}\to {\mathscr{C}}_\mathbb{Z}$, $(i,j)\mapsto(j,i)$ we get a Hausdorff left-continuous topology $\tau_p^l$ such that right translations are not continuous in $(\mathscr{C}_\mathbb{Z},\tau_p^l)$.
\end{example}

\begin{remark}\label{remark-3.6}
\begin{enumerate}
  \item We observe that for any integer $i_0$ the subset $\widetilde{L}_{i_0}=\left\{(i_0,i_0+i)\colon i\in\mathbb{Z}\right\}$ is not a subsemigroup of ${\mathscr{C}}_\mathbb{Z}$ because $(i_0,i_0-1)\cdot(i_0,i_0-1)=(i_0+1,i_0-1)\notin \widetilde{L}_{i_0}$.
  \item It is obvious that the topology $\tau_p^r$ ($\tau_p^l$) of the extended bicyclic semigroup ${\mathscr{C}}_\mathbb{Z}$ induces on the semigroup ${\mathscr{C}}_\mathbb{Z}^+$ (${\mathscr{C}}_\mathbb{Z}^-$) a Hausdorff non-discrete right-continuous (left-continuous) topology which is not left-continuous (right-continuous).
  %\item
\end{enumerate}
\end{remark}

The following example shows that the statement of Theorem~\ref{theorem-3.4} does not hold in the case when the band $E(S)$ of a subsemigroup $S$ of ${\mathscr{C}}_\mathbb{Z}^+$ is infinite.

\begin{example}\label{example-3.7}
We define $S=S_0\cup S_1$, where $S_0=\left\{(-i,-i)\in{\mathscr{C}}_\mathbb{Z}^+\colon i\in\mathbb{N}\right\}$ and $S_1=\left\{(0,k){\mathscr{C}}_\mathbb{Z}^+\colon k\in\omega\right\}$. Simple verifications show that $S$ is a commutative subsemigroup of ${\mathscr{C}}_\mathbb{Z}^+$. Moreover, $S_1$ is an ideal of $S$ such that $(-i,-i)\cdot (0,k)=(0,k)$ for all $(-i,-i)\in S_0$ and $(0,k)\in S_1$.

We define a topology $\tau$ on $S$ in the following way. All points of the form $(-i,-i)$, $i\in\mathbb{N}$, are isolated in $(S,\tau)$. Fix an arbitrary prime integer $p$. The family $\mathscr{B}_\tau=\left\{U_n(0,k)\colon n\in\mathbb{N}\right\}$, where
\begin{equation*}
  U_n(0,k)=\left\{(0,k+p^ns)\in S_1\colon s\in\omega\right\},
\end{equation*}
is the base of the topology $\tau$ at the point $(0,k)\in S_1$. Simple verifications show that $\tau$ is a Hausdorff topology on $S$.

Since $(-i,-i)\cdot (-j,-j)=(\max\{-i,-j\},\max\{-i,-j\})$, $(-i,-i)\cdot U_n(0,k)= U_n(0,k)$, and $U_n(0,k)\cdot U_n(0,l)\subseteq U_n(0,k+l)$ for any $i,j,k\in\mathbb{N}$ and $k,l\in\omega$, the semigroup operation on $(S,\tau)$ is continuous.
\end{example}

By Lemma~\ref{lemma-3.3} the semigroup ${\mathscr{C}}_\mathbb{Z}^+$ contains infinitely many isomorphic copies of the semigroup $\mathscr{C}_{+}(a,b)$ and hence by Theorem~\ref{theorem-2.7} ${\mathscr{C}}_\mathbb{Z}^+$ has continuum many subsemigroups $\left\{S_\alpha\colon \alpha\in\mathfrak{c}\right\}$ such that for any $\alpha\in\mathfrak{c}$ the following conditions hold:
\begin{enumerate}
  \item[$(i)$] the band $E(S_\alpha)$ is isomorphic to $(\omega,\max)$;
  \item[$(ii)$] for any idempotent $(i,i)$ of $S_\alpha$ the set $S\setminus (S\cdot (i,i))$ is finite;
  \item[$(iii)$] every left-continuous Hausdorff topology on $S$ is discrete.
\end{enumerate}

Example~\ref{example-3.8} and Theorem~\ref{theorem-3.10} imply that ${\mathscr{C}}_\mathbb{Z}^+$ has continuum many subsemigroups $\left\{S_\alpha\colon \alpha\in\mathfrak{c}\right\}$ such that for any $\alpha\in\mathfrak{c}$ the following conditions hold:
\begin{enumerate}
  \item[$(i^\circ)$] the band $E(S_\alpha)$ is isomorphic to $(\mathbb{Z},\max)$;
  \item[$(ii^\circ)$] for any idempotent $(i,i)$ of $S_\alpha$ the set $S\setminus (S\cdot (i,i))$ is infinite;
  \item[$(iii^\circ)$] every left-continuous Hausdorff topology on $S$ is discrete.
\end{enumerate}

\begin{example}\label{example-3.8}
Latter we shall assume that $X=\left\{x_i\colon i\in\omega\right\}$ where $\left\{x_i\right\}_{i\in\omega}$ is a steadily increasing sequence in $\omega$. For any positive integer $k$ we define
\begin{equation*}
  {\mathscr{C}}_\mathbb{Z}^+(k,X)=\left\{(-i,-i)\colon i\in\omega\right\}\cup \left\{(x_i,x_i)\colon i\in\omega\right\}\cup \left\{(m,n)\in\mathscr{C}^{+}_{\mathbb{Z}}[0]\colon m\geqslant k\right\}.
\end{equation*}
Simple verification show that ${\mathscr{C}}_\mathbb{Z}^+(k,X)$ is a subsemigroup of ${\mathscr{C}}_\mathbb{Z}^+$ for any positive integer $k$.
\end{example}

\begin{lemma}\label{lemma-3.9}
Let $\left\{x_i\right\}_{i\in\omega}$ and $\left\{y_i\right\}_{i\in\omega}$ be steadily increasing sequences in $\omega$ and let $X=\left\{x_i\colon i\in\omega\right\}$ and $Y=\left\{y_i\colon i\in\omega\right\}$. Then the subsemigroups ${\mathscr{C}}_\mathbb{Z}^+(1,X)$ and ${\mathscr{C}}_\mathbb{Z}^+(1,Y)$ of $\mathscr{C}^{+}_{\mathbb{Z}}$ are isomorphic if and only if $x_i=y_i$ for all $i\in\omega$.
\end{lemma}

\begin{proof}
The implication $(\Leftarrow)$ is trivial.

$(\Rightarrow)$  Suppose that $\mathfrak{I}\colon {\mathscr{C}}_\mathbb{Z}^+(1,X)\to{\mathscr{C}}_\mathbb{Z}^+(1,Y)$ is an isomorphism. Put $S_X={\mathscr{C}}_\mathbb{Z}^+(1,X)\cap \mathscr{C}^{+}_{\mathbb{Z}}[0]$ and $S_Y={\mathscr{C}}_\mathbb{Z}^+(1,Y)\cap \mathscr{C}^{+}_{\mathbb{Z}}[0]$. Simple verifications show that $S_X$ and $S_Y$ are subsemigroups of ${\mathscr{C}}_\mathbb{Z}^+(1,X)$ and ${\mathscr{C}}_\mathbb{Z}^+(1,Y)$, respectively. Moreover $S_X$ is isomorphic to $\mathscr{C}_{+1}^{X}(a,b)$ and $S_Y$ is isomorphic to $\mathscr{C}_{+1}^{Y}(a,b)$ by the mapping $(m,n)\mapsto b^ma^n$. The semigroup operation of ${\mathscr{C}}_\mathbb{Z}^+(1,X)$ and ${\mathscr{C}}_\mathbb{Z}^+(1,Y)$ implies that the equation $z\cdot (i,i+p)=(i,i+p)$ (in ${\mathscr{C}}_\mathbb{Z}^+(1,X)$ and ${\mathscr{C}}_\mathbb{Z}^+(1,Y)$), where $i\in\omega$ and $p\in\mathbb{N}$, has finitely many solutions, we conclude that the restriction $\mathfrak{I}_{S_X}\colon S_X\to S_Y$ is an isomorphism, and hence by Lemma~\ref{lemma-2.6} we get that $x_i=y_i$ for all $i\in\omega$.
Then the semillattices ${\mathscr{C}}_\mathbb{Z}^+(1,X)\setminus S_X$ and ${\mathscr{C}}_\mathbb{Z}^+(1,Y)\setminus S_Y$ are isomorphic. It is obvious that ${\mathscr{C}}_\mathbb{Z}^+(1,X)\setminus S_X$ and ${\mathscr{C}}_\mathbb{Z}^+(1,Y)\setminus S_Y$ are isomorphic to the semilattice $(\omega,\min)$ by the mapping $(-i,-i)\mapsto i$. Since the natural partial order on the semilattice $(\omega,\min)$ is complete, i.e., $(\omega,\min,\preccurlyeq)$, we have that $(-i,-i)\mathfrak{I}=(-i,-i)$ for any $i\in\omega$.
\end{proof}

We observe that for any steadily increasing sequence $\left\{x_i\right\}_{i\in\omega}$ in $\omega$ the semigroup ${\mathscr{C}}_\mathbb{Z}^+(1,X)$, where $X=\left\{x_i\colon i\in\omega\right\}$, satisfies the conditions $(i^\circ)$, $(ii^\circ)$, and $(iii^\circ)$. Then Lemma~\ref{lemma-3.9} implies the following theorem.

\begin{theorem}\label{theorem-3.10}
The semigroup ${\mathscr{C}}_\mathbb{Z}^+$ contains continuum many non-isomorphic subsemigroups $S_\alpha$ such that every left-continuous Hausdorff topology on $S_\alpha$ is discrete and $S_\alpha$ satisfies the conditions $(i^\circ)$, $(ii^\circ)$, and $(iii^\circ)$.
\end{theorem}

\section{Embedding of the semigroup ${\mathscr{C}}_\mathbb{Z}^+$ into a compact topological semigroup}

The following example of the closure of the extended bicyclic semigroup ${\mathscr{C}}_\mathbb{Z}$ is presented in \cite{Fihel-Gutik=2011}.

\begin{example}[{\cite[Example~5]{Fihel-Gutik=2011}}]\label{example-4.1}
By $G_1(1)$ and $G_0$ we denote the additive groups of integers. We denote by $(n)^1$ and $(m)^0$ the elements of the groups $G_1(1)$ and $G_0$, respectively, $n,m\in\mathbb{Z}$. This means that $(m)^1\cdot (n)^1=(m+n)^1$ and $(m)^0\cdot (n)^0=(m+n)^0$ for any $m,n\in\mathbb{Z}$. Put $S=G_1(1)\sqcup  {\mathscr{C}}_\mathbb{Z}\sqcup G_0$ and extent the semigroup operation from ${\mathscr{C}}_\mathbb{Z}$, $G_1(1)$ and $G_0$ in the following way:
\begin{align*}
  (n)^1\cdot(i,j)&=(-n+i,j)\in\mathscr{C}_\mathbb{Z}; \\
  (i,j)\cdot(n)^1&=(i,j+n)\in\mathscr{C}_\mathbb{Z}; \\
  (n)^1\cdot(m)^0&=(n+m)^0\in G_0; \\
  (m)^0\cdot(n)^1&=(m+n)^0\in G_0; \\
  (m)^0\cdot(i,j)&=(m+j-i)^0\in G_0; \\
  (i,j)\cdot(m)^0&=(m+j-i)^0\in G_0,
\end{align*}
for any $(n)^1\in G_1(1)$, $(i,j)\in\mathscr{C}_\mathbb{Z}$, and $(m)^0\in G_0$.

The topology $\tau$ on $S$ is determined in the following way:
\begin{enumerate}
  \item all elements $(i,j)$ of the extended bicyclic semigroup ${\mathscr{C}}_\mathbb{Z}$ are isolated in $(S,\tau)$;
  \item the family $\mathscr{B}((n)^1)=\left\{U_p((n)^1)\colon p\in\mathbb{N}\right\}$, where $U_p((n)^1)=\left\{(n)^1\right\}\cup\left\{(-q,-q+n)\in\mathscr{C}_\mathbb{Z}\colon q\geqslant p\right\}$, is a base of the topology $\tau$ at the point $(n)^1\in G_1(1)$;
  \item the family $\mathscr{B}((m)^0)=\left\{V_p((m)^0)\colon p\in\mathbb{N}\right\}$, where
      \begin{equation*}
      V_p((m)^0)=
      \left\{
        \begin{array}{ll}
          \left\{(m)^0\right\}\cup\left\{(q,q+m)\in\mathscr{C}_\mathbb{Z}\colon q\geqslant p\right\}, & \hbox{if~} m\geqslant 0;\\
          \left\{(m)^0\right\}\cup\left\{(q-m,q)\in\mathscr{C}_\mathbb{Z}\colon q\geqslant p\right\}, & \hbox{if~} m\leqslant 0,
        \end{array}
      \right.
      \end{equation*}
      is a base of the topology $\tau$ at the point $(m)^0\in G_0$.
\end{enumerate}

Then $(S,\tau)$ is a Hausdorff locally compact topological inverse semigroup \cite{Fihel-Gutik=2011}.
\end{example}

\begin{example}\label{example-4.2}
Let $G_1(1)$ and $G_0$ be the groups which is defined in Example~\ref{example-4.1}. We denote $G_1^+=\left\{(n)^1\in G_1(1)\colon n\geqslant 0\right\}$ and $G_0^+=\left\{(m)^0\in G_0\colon m\geqslant 0\right\}$. The semigroup operation $S$ implies that the subset $S^+=G_1^+\sqcup {\mathscr{C}}_\mathbb{Z}^+\sqcup G_0^+$ of $S$ with the induced from $S$ multiplication is a subsemigroup of $S$. Let $\tau^+$ be the topology on $S^+$ which is induced from $(S,\tau)$. Then $(S^+,\tau^+)$ is a Hausdorff topological semigroup. Since $(S^+,\tau^+)$ is a closed subspace of $(S,\tau)$, Theorem~3.3.8 of \cite{Engelking=1989} implies that the space $(S^+,\tau^+)$ is locally compact.

We define $S^+_{\mathscr{O}}$ to be the semigroup $S^+$ with the adjoined zero $\mathscr{O}$, i.e., $\mathscr{O}\cdot s=s\cdot \mathscr{O}=\mathscr{O}\cdot\mathscr{O}=\mathscr{O}$ for any $s\in S^+$. We define a topology $\tau^+_{\mathscr{O}}$ on the semigroup $S^+_{\mathscr{O}}$ as follows. We extent the topology $\tau^+$ onto the semigroup $S^+_{\mathscr{O}}$ in the following way: the family $\mathscr{B}^+(\mathscr{O})=\left\{W_p(\mathscr{O})\colon p\in\mathbb{N}\right\}$, where
\begin{equation*}
  W_p(\mathscr{O})=\left\{\mathscr{O}\right\}\sqcup\left\{(n)^1\in G_1(1)\colon n\geqslant p\right\}\sqcup\left\{(i,j)\in{\mathscr{C}}_\mathbb{Z}^+\colon j-i\geqslant p\right\} \sqcup\left\{(m)^0\in G_0\colon m\geqslant p\right\},
\end{equation*}
is the base of the topology $\tau^+_{\mathscr{O}}$ at the point $\mathscr{O}$. It is obvious that $\tau^+_{\mathscr{O}}$ is a compact Hausdorff topology on $S^+_{\mathscr{O}}$. Since the semigroup $S^+_{\mathscr{O}}$ is countable, Theorems~3.1.21 and~4.2.8 from \cite{Engelking=1989} imply that the space $(S^+_{\mathscr{O}},\tau^+_{\mathscr{O}})$ is metrizable.
\end{example}

\begin{proposition}\label{proposition-4.3}
$(S^+_{\mathscr{O}},\tau^+_{\mathscr{O}})$ is a topological semigroup.
\end{proposition}

\begin{proof}
$(S^+,\tau^+)$ is a Hausdorff topological semigroup and $(S^+,\tau^+)$ is an open subspace of the topological space $(S^+_{\mathscr{O}},\tau^+_{\mathscr{O}})$, it is complete to show that the semigroup operation is continuous on $(S^+_{\mathscr{O}},\tau^+_{\mathscr{O}})$ in the following seven cases.
\begin{enumerate}
  \item In the case $(n)^1\cdot \mathscr{O}$, where $(n)^1\in G_1(1)$, we have that $U_{p_1}((n)^1)\cdot W_{p_2}(\mathscr{O})\subseteq W_{n+p_2}(\mathscr{O})$ for any non-negative integer $n$ and any positive integers $p_1$ and $p_2$.
  \item In the case $ \mathscr{O}\cdot(n)^1$, where $(n)^1\in G_1(1)$, we obtain that $W_{p_1}(\mathscr{O})\cdot U_{p_2}((n)^1)\subseteq W_{p_1+n}(\mathscr{O})$ for any non-negative integer $n$ and any positive integers $p_1$ and $p_2$.
  \item In the case $(i,j)\cdot \mathscr{O}$, where $(i,j)\in {\mathscr{C}}_\mathbb{Z}^+$, we get that $(i,j)\cdot W_{p}(\mathscr{O})\subseteq W_{p+j-i}(\mathscr{O})$ for any positive integer $p$.
  \item In the case $\mathscr{O}\cdot(i,j)$, where $(i,j)\in {\mathscr{C}}_\mathbb{Z}^+$, we have that $W_{p}(\mathscr{O})\cdot(i,j)\subseteq W_{p+j-i}(\mathscr{O})$ for any positive integer $p$.
  \item In the case $(m)^0\cdot \mathscr{O}$, where $(m)^0\in G_0$, we obtain that $V_{p_1}((m)^0)\cdot W_{p_2}(\mathscr{O})\subseteq W_{m+p_2}(\mathscr{O})$ for any non-negative integer $m$ and any positive integers $p_1$ and $p_2$.
  \item In the case $\mathscr{O}\cdot (m)^0$, where $(m)^0\in G_0$, we get that $W_{p_1}(\mathscr{O})\cdot V_{p_2}((m)^0)\subseteq W_{p_1+m}(\mathscr{O})$ for any non-negative integer $m$ and any positive integers $p_1$ and $p_2$.
  \item In the case $\mathscr{O}\cdot\mathscr{O}$ we have that $W_{p_1}(\mathscr{O})\cdot W_{p_2}(\mathscr{O})\subseteq W_{p_1+p_2}(\mathscr{O})$ for any positive integers $p_1$ and $p_2$.
\end{enumerate}

The above arguments imply that the semigroup operation is continuous on $(S^+_{\mathscr{O}},\tau^+_{\mathscr{O}})$.
\end{proof}

The definition of the topology $\tau^+_{\mathscr{O}}$ implies that $I_0=G_0^+\sqcup \{\mathscr{O}\}$ is a compact ideal of the Hausdorff compact topological semigroup $(S^+_{\mathscr{O}},\tau^+_{\mathscr{O}})$. By the Lawson-Madison theorem (see \cite{Lawson-Madison=1971} or \cite[Theorem~1.57]{Carruth-Hildebrant-Koch=1983}) the quotient semigroup $S^+_{\mathscr{O}}/{I_0}$ with the quotient topology $\tau_{\mathbf{q}}$ is a Hausdorff compact topological semigroup. Since ${\mathscr{C}}_\mathbb{Z}^+$ is a dense subsemigroup of  $(S^+_{\mathscr{O}},\tau^+_{\mathscr{O}})$ and ${\mathscr{C}}_\mathbb{Z}^+\cap I_0=\varnothing$ we get the following corollary.

\begin{corollary}\label{corollary-4.4}
The semigroup ${\mathscr{C}}_\mathbb{Z}^+$ is a dense subsemigroup of the Hausdorff compact topological semigroup  $(S^+_{\mathscr{O}}/{I_0},\tau_{\mathbf{q}})$.
\end{corollary}

Since the closure of a subsemigroup $S$ of a Hausdorff compact topological semigroup $K$ is a compact  topological semigroup, Proposition~\ref{proposition-4.3} implies the following two corollaries.

\begin{corollary}\label{corollary-4.5}
Every subsemigroup of  semigroup ${\mathscr{C}}_\mathbb{Z}^+$ is a dense subsemigroup of a Hausdorff compact topological semigroup.
\end{corollary}

\begin{corollary}\label{corollary-4.6}
On the semigroup $\mathscr{C}_{+}(a,b)$ with adjoined zero admits a Hausdorff compact semigroup topology.
\end{corollary}

\begin{remark}
\begin{enumerate}
  \item Similar statements to Proposition~\ref{proposition-4.3} and Corollaries~\ref{corollary-4.4} and \ref{corollary-4.5} hold for the semigroup ${\mathscr{C}}_\mathbb{Z}^-$. Also, the statement of Corollary~\ref{corollary-4.6} holds for the semigroup $\mathscr{C}_{-}(a,b)$.
  \item In \cite{Gutik=2015} it is proved that a Hausdorff locally compact semitopological bicyclic semigroup with adjoined zero $\mathscr{C}^0$ is either compact or discrete. Similar statement holds for semitopological interassociates of the bicyclic monoid \cite{Gutik-Maksymyk=2016}. But the extended bicyclic semigroup  with adjoined zero ${\mathscr{C}}_\mathbb{Z}^0$ admits continuum many shift-continuous locally compact topologies \cite{Gutik-Maksymyk=2019} and similar statements hold for the semigroups $\mathscr{C}_{+}(a,b)$ and $\mathscr{C}_{-}(a,b)$ \cite{Gutik=2024}.
\end{enumerate}
\end{remark}
%%%%%%%%%%%%%%%%%%%%%%%%%%%%%%%%%%%%%%%%%%%%%%%%%%%%%
\section*{\textbf{Acknowledgements}}

The author acknowledges Alex Ravsky for his comments and suggestions.

\end{document}